\documentclass[11pt]{amsart}

\usepackage{amsmath}
\usepackage{amssymb}
\usepackage{graphicx}
\usepackage{amscd}

\newtheorem{theorem}{Theorem}[section]
\newtheorem{corollary}[theorem]{Corollary}
\newtheorem{lemma}[theorem]{Lemma}
\newtheorem{proposition}[theorem]{Proposition}
\theoremstyle{definition}
\newtheorem{definition}[theorem]{Definition}
\newtheorem{example}[theorem]{Example}
\newtheorem{remark}[theorem]{Remark}

\numberwithin{equation}{section}

\newcommand{\R}{\mathbb R}
\newcommand{\N}{\mathbb N}
\newcommand{\B}{\mathbb B}

\newcommand{\X}{\mathbb X}
\newcommand{\Y}{\mathbb Y}

\newcommand{\nullv}{\mathbf{0}}
\newcommand{\Sfer}{\mathbb S}

\newcommand{\extr}{{\rm ext}\, }
\newcommand{\cl}{{\rm cl}\, }
\newcommand{\inte}{{\rm int}\, }
\newcommand{\fr}{{\rm bd}\, }

\newcommand{\diamet}{{\rm diam}\, }

\newcommand{\Lagr}{{\rm L}}
\newcommand{\gap}{{\rm gap}\, }
\newcommand{\Coneone}{{\rm C}^{1,1}}

\newcommand{\Uniconvc}{{\mathcal UC}^2_c}
\newcommand{\Uniconv}{{\mathcal UC}^2}

\newcommand{\dif}[1]{{\rm D}{#1}}
\newcommand{\immap}[1]{f_{{\mathcal P},{#1}}}

\newcommand{\ball}[2]{{\rm B}\left(#1, #2\right)}
\newcommand{\dist}[2]{{\rm dist}\left(#1,#2\right)}
\newcommand{\reg}[2]{{\rm reg}({#1},{#2})}

\newcommand{\lip}[2]{{\rm lip}({#1},{#2})}
\newcommand{\der}[2]{{\rm D}{#1}({#2})}
\newcommand{\ncone}[2]{{\rm N}({#1},{#2})}

\title[An extension of Polyak convexity principle with application]{An extension of the
Polyak convexity principle with application to nonconvex optimization}

\author[A. Uderzo]{Amos Uderzo}

\address[A. Uderzo]{Dept. of Mathematics and Applications,
University of Milano-Bicocca, Italy}
\email{{\tt amos.uderzo@unimib.it}}



\date{\today}

\begin{document}

\begin{abstract}
The main problem considered in the present paper is to single out classes
of convex sets, whose convexity property is preserved under nonlinear
smooth transformations. Extending an approach due to B.T. Polyak, the present
study focusses on the class of uniformly convex subsets of Banach spaces.
As a main result, a quantitative condition linking the modulus of convexity
of such kind of set, the regularity behaviour around a point of a nonlinear
mapping and the Lipschitz continuity of its derivative is established, which
ensures the images of uniformly convex sets to remain uniformly convex.
Applications of the resulting convexity principle to the existence of solutions,
their characterization and to the Lagrangian duality theory in constrained
nonconvex optimization are then discussed.
\end{abstract}

\maketitle



\section{Introduction}

In many fields of mathematics, persistence phenomena of specific
geometrical properties under various kind of transformations have been often
a subject of interest and study. Transformations, when possible formalized
by mappings acting among spaces, sometimes have been classified on the basis
of features in a structure that they can preserve (whence the very term
``morphism'').
Convexity is a geometrical property which emerged in ancient times, at
the very beginning of geometry, and since then remained essentially unchanged for
almost two millennia and half. This happened by virtue of the great variety
of successful applications that it found in many different areas.
In particular, the relevant role played by convexity in optimization
and control theory is widely recognized. This led to develop a branch
of mathematics, called convex analysis, that elected convexity
as its main topic of study. In spite of such an interest and motivations,
not much seems to be known up to now about phenomena
of persistence of convexity under nonlinear transformations. Yet, advances
in this direction would have a certain impact on the analysis of
optimization problems.
Historically, the first results somehow connected with the issue at the
study relate to the numerical range of quadratic mappings (namely,
mappings whose components are quadratic forms) and can be found in
\cite{Dine41} (see also \cite{Poly98}). A notable step ahead was made when the preservation
of convexity of small balls under smooth regular transformations
between Hilbert spaces was established by B.T. Polyak (see \cite{Poly01}).
After that, some other contributions to understanding the phenomenon
in a similar context were given by \cite{BacSam09,BoEmKo04,Dyma16,Reis07}.
Various applications of it to topics in linear algebra, optimization and control
theory are presented in \cite{Poly98,Poly01,Poly01b,Poly03,Reis07}

In the present paper, by following the approach introduced by B.T. Polyak,
the study of classes of sets with persistent convexity properties
is carried on. More precisely, the analysis here proposed focusses on the class
of uniformly convex subsets of certain Banach spaces. An interest in
similar classes of sets, in connection with the problem under study, appears
already in \cite{Poly01}, where strongly convex sets are actually mentioned.
This seems to be rather natural, inasmuch as elements of such classes
share the essential geometrical features of balls in a Hilbert space:
nonempty interior, boundedness and, what plays a crucial role,
a uniform rotundity, which implies a boundary consisting of extreme points only.
The feature last mentioned is captured and quantitatively expressed by the notion
of modulus of convexity of a set. In developing the Polyak's approach,
the main idea behind the investigations exposed in the paper is that,
if the modulus of convexity of a given
set matches the smoothness and the regularity property of a given
nonlinear mapping, then the persistence of convexity under that mapping
can be guaranteed. The understanding of such a fundamental relation
between quantitative aspects of the convexity property for a set
and the quantitative regularity behaviour of a mapping acting on it
should shed light on the general phenomenon under study.
Concretely, this leads to enrich the class of sets interested by
the phenomenon. In turn, since the persistence of convexity under
nonlinear transformations is at the origin of a certain qualification
(in terms of solution existence and characterization)
observed in optimization problems with possibly nonconvex data,
the result here established allows one to enlarge the class of
problems for which the consequent benefits can be expected.

The contents of the paper are arranged in the next sections as follows.
In Section \ref{Sec:2}, the notion of modulus of convexity of a set
and of uniformly convexity are recalled, along with several examples
and related facts, useful for the subsequent analysis. Besides, the
regularity behaviour of a nonlinear smooth mapping, namely its openness
at a linear rate, is entered as a crucial tool, along with the related
exact bound.
In Section \ref{Sec:3}, the main result of the paper, which is an
extension of the aforementioned convexity principle due to B.T. Polyak,
is established and some of its features are discussed.
In Section \ref{Sec:4}, some applications of the main result to nonconvex
constrained optimization problems are provided.


\section{Notations and preliminaries}    \label{Sec:2}

The basic notations in use throughout the paper are as follows.
$\R$ denotes the real number set. Given a metric space $(X,d)$,
an element $x_0\in X$ and $r\ge 0$, $\ball{x_0}{r}=\{x\in X:\ d(x,x_0)
\le r\}$ denotes the (closed) ball with center $x_0$ and radius $r$.
In particular, in a Banach space, the unit ball centered at the
null  vector will be indicated by $\B$, whereas the unit sphere by
$\Sfer$. The distance of $x_0\in X$ from a set $S\subseteq X$
is denoted by $\dist{x_0}{S}$. If $S\subseteq X$, $\ball{S}{r}=
\{x\in X:\ \dist{x}{S}\le r\}$ denotes the (closed) $r$-enlargement
of $S$. The diameter of a set $S\subseteq
X$ is defined as $\diamet S=\sup\{d(x_1,x_2):\ x_1,\, x_2\in S\}$.
By $\inte S$, $\cl S$ and $\fr S$ the topological interior, the
closure and the boundary of a set $S$ are marked, respectively.
If $S$ is a subset of a Banach space $(\X,\|\cdot\|)$, $\extr S$
denotes the set of all extreme points of $S$, in the sense
of convex analysis, $\nullv$ stands for the null element of $\X$
and $[x_1,x_2]$ denotes the closed line segment with endpoints
$x_1,\, x_2\in\X$. Given a function $h:X\longrightarrow Y$ between
metric spaces and a set $U\subseteq X$, $h$ is said to be
Lipschitz continuous on $U$ if there exists a constant $\ell>0$
such that
\begin{equation}     \label{in:defLip}
  d(h(x_1),h(x_2))\le\ell d(x_1,x_2),\quad \forall x_1,\, x_2\in U.
\end{equation}
The infimum over all values $\ell$ making the last
inequality satisfied on $U$ is called exact bound of Lipschitz continuity
of $h$ on $U$ and is denoted by $\lip{h}{U}$, i.e.
$$
  \lip{h}{U}=\inf\{\ell\ge 0:\ \hbox{ inequality $(\ref{in:defLip})$ holds}\}.
$$
The Banach space of all bounded linear operators between the Banach spaces $\X$
and $\Y$, equipped with the operator norm, is denoted by
$(\mathcal{L}(\X,\Y),\|\cdot\|_\mathcal{L})$. If, in particular,
it is $\Y=\R$, the simpler notation $(\X^*,\|\cdot\|_*)$ is used.
The null vector in a dual space is marked by $\nullv^*$, whereas
the unit sphere by $\Sfer^*$, with $\langle\cdot,\cdot\rangle$
marking the duality pairing a space and its dual.
Given a mapping $f:\Omega\longrightarrow\Y$, with $\Omega$ open
subset of $\X$, and $x_0\in\Omega$, the Gat\^eaux derivative of $f$ at $x_0$ is
denoted by $\der{f}{x_0}$. If $f$ is Gat\^eaux differentiable
at each point of $\Omega$ and the mapping $\dif{f}:\Omega\longrightarrow
\mathcal{L}(\X,\Y)$ is Lipschitz continuous on $\Omega$, $f$ is said
to be of class $\Coneone(\Omega)$.

\begin{remark}     \label{rem:smoothfacts}
(i) In view of a subsequent employment, let us recall that, whenever
$f:\Omega\longrightarrow\Y$ is a mapping of class $\Coneone(\Omega)$
between Banach spaces, with $\Omega$ open subset of $\X$ and $x_1,\,
x_2\in\Omega$ are such that $[x_1,x_2]\subseteq\Omega$, the
following estimate holds true (see, for instance, \cite[Lemma 2.7]{Uder13})
\begin{equation}    \label{in:2ndordest}
    \left\|{f(x_1)+f(x_2)\over 2}-f\left({x_1+x_2\over 2}\right)
    \right\|\le {\lip{\dif{f}}{\Omega}\over 8}\|x_1-x_2\|^2,
\end{equation}
where $\lip{\dif{f}}{\Omega}$ denotes the exact bound of
Lipschitz continuity of $\dif{f}$ on $\Omega$.

(ii) It is not difficult to see that, if $S\subseteq\Omega$ is
a bounded set, i.e. $\diamet S<+\infty$, and $f\in\Coneone
(\Omega)$, then it must be
$$
   \sup_{x\in S}\|\der{f}{x}\|_\mathcal{L}<+\infty.
$$
Furthermore, if in addition $S$ is convex, then letting
$\beta_S=\sup_{x\in S}\|\der{f}{x}
\|_\mathcal{L}$, as an immediate consequence of the
mean-value theorem, one obtains
$$
   \diamet f(S)\le\beta_S\diamet S,
$$
that is $f(S)$ is bounded too.
\end{remark}

\subsection{Uniformly convex sets}

\begin{definition}      \label{def:uniconv}
(i) Let $S\subseteq\X$ be a nonempty, closed and convex subset
of a real Banach space. The function $\delta_S:[0,\diamet S)\longrightarrow
[0,+\infty)$ defined by
$$
    \delta_S(\epsilon)=\sup\left\{\delta\ge 0:\ \ball{{x_1+x_2\over 2}}{\delta}
              \subseteq S,\ \forall x_1,\, x_2\in S:\ \|x_1-x_2\|=\epsilon\right\}
$$
is called {\it modulus of convexity} of the set $S$. Whenever the value of
$\diamet S$ is attained at some pair $x_1,\, x_2\in S$, the function
$\delta_S$ will be meant to be naturally extended to $[0,\diamet S]$.

(ii) After \cite{Poly66}, a nonempty, closed and convex set $S\subseteq\X$, with $S\ne\X$, is said
to be {\it uniformly convex} provided that
$$
    \delta_S(\epsilon)>0,\quad\forall \epsilon\in
   \left\{ \begin{array}{ll}
                  (0,\diamet S], &\quad\hbox{if $\diamet S$ is attained on $S$}, \\
                  (0,\diamet S), &\quad\hbox{otherwise.}
                \end{array}   \right.
$$
\end{definition}

Since $\diamet S$ vanishes if $S$ is a singleton, Definition
\ref{def:uniconv} (ii) does not exclude such kind of convex sets.
Nevertheless, as singletons are of minor interest in connection
with the problem at the issue, henceforth a uniformly convex set
will be always assumed to contain at least two distinct points.

\begin{example}      \label{ex:ucset}
(i) Balls in a uniformly convex Banach space may be viewed as a paradigma
for the notion of uniform convexity for sets. Recall that, after
\cite{Clar36}, a Banach space $(\X,\|\cdot\|)$ is said to be
{\it uniformly convex} (or to have a uniformly convex norm) if
$$
    \delta_{\X}(\epsilon)=\inf\left\{1-\left\|{x_1+x_2\over 2}\right\|:\
                   x_1,\, x_2\in\B,\ \|x_1-x_2\|=\epsilon\right\}>0,
    \forall\epsilon\in (0,2].
$$
The function $\delta_\X$ is called modulus of convexity of the space
$(\X,\|\cdot\|)$. In fact, it is possible to prove that
$$
   \delta_\B(\epsilon)=\delta_\X(\epsilon),\quad\forall
   \epsilon\in (0,2].
$$
Such classes of Banach spaces as $l^p$ and $L^p$, with
$1<p<\infty$, are known to consist of uniformly convex spaces. In
particular, every Hilbert space is uniformly convex.
Since every uniformly convex Banach space must be reflexive
(according to the Milman-Pettis Theorem),
the spaces $l^1,\, L^1,\, L^\infty,\, C([0,1])$ and $c_0$ fail
to be. For $p\ge 2$, the exact expression of the modulus of
convexity of the spaces $l^p$ and $L^p$ is given by
$$
    \delta_{l^p}(\epsilon)=\delta_{L^p}(\epsilon)=
    1-\left[1-\left({\epsilon\over 2}\right)^p\right]^{1/p},
   \quad\forall\epsilon\in (0,2].
$$
For more details on uniformly convex Banach spaces and properties
of their moduli the reader may refer to
\cite{Dies75,FaHaHaMoPeZi01,Milm71}.
A useful remark enlightening the connection between the notions
of uniform convexity for sets and uniform convexity of Banach spaces
can be found in \cite[Theorem 2.3]{BalRep09}:
a Banach space can contain a closed uniformly convex set iff it admits an
equivalent uniformly convex norm. Such class of Banach spaces have been
characterized in terms of superreflexivity in \cite{Enfl72}.
Throughout the present paper, the Banach space $(\X,\|\cdot\|)$ will be
supposed to be equipped with a uniformly convex norm.

(ii) After \cite{Polo94,Polo96}, given a positive real $r$,
a subset $S\subseteq\X$ of a Banach space is said to be
{\it $r$-convex} (or {\it strongly convex} of radius $r$)
if there exists $M\subseteq\X$, with $M\ne\X$, such that
$$
  S=\bigcap_{x\in M}\ball{x}{r}\ne\varnothing.
$$
It is readily seen that, if a Banach space $(\X,\|\cdot\|)$ is
uniformly convex with modulus $\delta_\X$, then any strongly
convex set $S\subseteq\X$ with radius $r$ is uniformly convex
and its modulus of convexity satisfies the relation
\begin{eqnarray}      \label{in:ucmscs}
    \delta_S(\epsilon)\ge r\delta_\X\left({\epsilon\over r}\right),
    \quad\forall \epsilon\in (0,\diamet S).
\end{eqnarray}

(iii) Let $\theta:[0,+\infty)\longrightarrow [0,+\infty)$ be an
increasing function vanishing only at $0$. Recall that, according
to \cite{Zali02}, a function $\varphi:\X\longrightarrow\R$ is said
to be {\it uniformly convex with modulus $\theta$} if it holds
\begin{eqnarray*}
   \varphi(tx_1+(1-t)x_2)\le t\varphi(x_1)+(1-t)\varphi(x_2)
   -t(1-t)\theta(\|x_1-x_2\|),\\ \forall x_1,\, x_2\in\X,
   \,\forall t\in [0,1].
\end{eqnarray*}
If, in particular, it is $\theta(s)=\kappa s^2$, a uniformly convex
function with such a modulus is called {\it strongly convex}.
Sublevel sets of Lipschitz continuous uniformly convex functions
are uniformly convex sets. More precisely, given $\alpha>0$, if
$\varphi$ is Lipschitz continuous on $\X$, with exact bound
$\lip{\varphi}{\X}>0$, then the set $[\varphi\le\alpha]=\{x\in\X:
\ \varphi(x)\le\alpha\}$ turns out to be uniformly convex with
modulus
\begin{eqnarray}    \label{in:ucmucfunct}
    \delta_{[\varphi\le\alpha]}(\epsilon)\ge
    {\theta(\epsilon)\over 4\lip{\varphi}{\X}},
    \quad\forall\epsilon \in (0,\diamet [\varphi\le\alpha]).
\end{eqnarray}
Indeed, fixed $\epsilon\in (0,\diamet [\varphi\le\alpha])$,
take $x_1,\, x_2\in [\varphi\le\alpha]$, with $x_1\ne x_2$ and
$\|x_1-x_2\|=\epsilon$, and set $\bar x={1\over 2}(x_1+x_2)$.
By the uniform convexity of $\varphi$ with modulus $\theta$
one has
$$
   \varphi(\bar x)\le {\varphi(x_1)+\varphi(x_2)\over 2}-
   {\theta(\|x_1-x_2\|)\over 4}.
$$
Therefore, for an arbitrary $\eta>0$, by the Lipschitz continuity
of $\varphi$ on $\X$, one finds
\begin{eqnarray*}
   \varphi(x)&=& \varphi(x)-\varphi(\bar x)+\varphi(\bar x) \\
   &\le&
   (\lip{\varphi}{\X}+\eta){\theta(\epsilon)\over
   4(\lip{\varphi}{\X}+\eta)}+\alpha-
   {\theta(\epsilon)\over 4}\le\alpha,
\end{eqnarray*}
for every $x\in\ball{\bar x}{{\theta(\epsilon)\over 4
(\lip{\varphi}{\X}+\eta)}}$. Thus, it results in
$$
  \ball{\bar x}{{\theta(\epsilon)\over 4
  (\lip{\varphi}{\X}+\eta)}}\subseteq
  [\varphi\le\alpha],
$$
so
$$
  \delta_{[\varphi\le\alpha]}(\epsilon)\ge{\theta(\epsilon)
\over 4(\lip{\varphi}{\X}+\eta)}.
$$
The estimate in $(\ref{in:ucmucfunct})$ follows by arbitrariness
of $\eta$.
\end{example}

It is not difficult to see that, given two subsets $S_1$ and $S_2$
of $\X$, it is $\delta_{S_1\cap S_2}\ge\min\{\delta_{S_1},\, \delta_{S_2}\}$.
Therefore, the class of uniformly convex sets is closed under finite
intersection. In contrast, unlike the class of convex sets, this
class fails to be closed with respect to the Cartesian product.
It is worth noting that, as the intersection of balls may yield
a boundary with corners or a nonsmooth description, uniformly
convex sets may exhibit such kind of pathology.

In the next remark, some known facts about uniformly convex
sets are collected, which will be relevant to the subsequent
analysis.

\begin{remark}    \label{rem:strconvfacts}
(i) Every uniformly convex set, which does not coincide with the
entire space, is bounded (see \cite{BalRep09}).

(ii) Directly from Definition \ref{def:uniconv}, it follows
that every uniformly convex set has nonempty interior. This
fact entails that, while uniformly convex subsets are compact
if living in finite-dimensional spaces, they can not be so
in infinite-dimensional Banach spaces.

(iii) As a consequence of Definition \ref{def:uniconv},
if any uniformly convex set $S$ admits a modulus of convexity
of power type $2$, i.e. such that
\begin{eqnarray}   \label{in:qgc}
    \delta_S(\epsilon)\ge c\epsilon^2,
    \quad\forall \epsilon\in (0,\diamet S),
\end{eqnarray}
for some $c>0$, then it fulfils the following property:
for every $\tilde c\in (0,c)$ it holds
$$
    \ball{{x_1+x_2\over 2}}{\tilde c\|x_1-x_2\|^2}\subseteq S,
    \quad\forall x_1,\, x_2\in S.
$$
It is worth noting that this happens for the balls in
any Hilbert space or in the Banach spaces $l^p$ and $L^p$,
with $1<p<2$, where the following estimate is known to hold
$$
  \delta_{l^p}(\epsilon)=\delta_{L^p}(\epsilon)>
  {p-1\over 8}\epsilon^2,\quad\forall\epsilon\in (0,2]
$$
(see, for instance, \cite{Milm71}). Such a subclass of uniformly
convex sets will play a prominent role in the main result of
the paper.

(iv) For every uniformly convex set $S$, a constant $\beta>0$ can
be proved to exist such that
$$
   \delta_S(\epsilon)\le \beta\epsilon^2,\quad\forall
  \epsilon\in (0,\diamet S)
$$
(see \cite{BalRep09}). Thus, a modulus of convexity of the power
$2$ is a maximal one.
\end{remark}

The next proposition provides a complete characterization of uniform
convexity for subsets of a finite-dimensional Euclidean space in terms
of extremality of their boundary points. Below, a variational proof
of this fact is provided.

\begin{proposition}     \label{pro:ucchar}
A convex compact subset $S\subseteq\R^n$, with nonempty interior,
is uniformly convex iff $\extr S=\fr S$.
\end{proposition}

\begin{proof}
Observe that by compactness of $S$, it is $\fr S\ne\varnothing$.
Actually, the Krein-Milman theorem ensures that
$\extr S\ne\varnothing$ also.
Clearly, it is $\extr S\subseteq\fr S$. To begin with,
assume that $S$ is uniformly convex. Take any $\bar x\in\fr S$.
If it were $\bar x\not\in\extr S$, then there would exist
$x_1,\, x_2\in S\backslash\{\bar x\}$, with $x_1\ne x_2$,
such that $\bar x={x_1+x_2\over 2}$. Observe that, as $\bar x
\in\fr S$, the inclusion $\ball{\bar x}{\delta}\subseteq S$
can be true only for $\delta=0$. Thus $\delta_S (\|x_1-x_2\|)
=0$, contradicting the fact that $S$ is uniformy convex.

Conversely, assume that the equality $\extr S=\fr S$ holds true.
Fix an arbitrary $\epsilon\in (0,\diamet S]$ (under the current hypotheses
the value $\diamet S$ is attained on $S$). Notice that, since $S$ is
compact, the set
$$
     S^2_\epsilon=\{(x_1,x_2)\in S\times S:\ \|x_1-x_2\|=\epsilon\}
$$
is still compact. Define the function $\vartheta:\R^n\times\R^n
\longrightarrow [0,+\infty)$ by setting
$$
    \vartheta(x_1,x_2)=\dist{{x_1+x_2\over 2}}{\R^n\backslash\inte S}.
$$
Since such a function is continuous on $\R^n\times\R^n$, it attains
its global minimum over $S^2_\epsilon$ at some point
$(\hat x_1,\hat x_2)\in S^2_\epsilon$, with $\hat x_1\ne\hat x_2$ as
$\|\hat x_1-\hat x_2\|=\epsilon$.
If it were $\vartheta(\hat x_1,
\hat x_2)=0$, then it would happen that
$$
    {\hat x_1+\hat x_2\over 2}\in\fr S.
$$
The last inclusion contradicts the fact that ${\hat x_1+\hat x_2\over 2}$
is an extreme point of $S$. Therefore, one deduces that $\vartheta
(\hat x_1,\hat x_2)>0$. As it is true that
$$
    \delta_S(\epsilon)=\min_{(x_1,x_2)\in S^2_\epsilon}
    \vartheta(x_1,x_2)>0,
$$
the requirement in Definition \ref{def:uniconv} (ii) turns out to be
satisfied. The arbitrariness of $\epsilon\in (0,\diamet S]$ completes
the proof.
\end{proof}

Proposition \ref{pro:ucchar} can not be extended to infinite-dimensional
spaces, where balls with $\extr\B=\fr\B$ can exist, yet failing
to be uniformly convex (see \cite{Dies75}).

\subsection{Openness at a linear rate}

In the next definition, some notions and related results are
recalled, which describe quantitatively a certain surjective behaviour
of a mapping. Such a local property, in a synergical interplay
with other features ($\Coneone$-smoothness and uniform convexity)
of the involved objects, allows one to achieve the main result
in the paper.

\begin{definition}     \label{def:lopmap}
Let $f:X\longrightarrow Y$ be a mapping between two metric spaces
and $x_0\in X$. The mapping $f$ is said to be
{\it open at a linear rate around} $x_0$ if there exist
positive reals $\delta$, $\zeta$ and $\sigma$ such that
\begin{equation}   \label{in:pointlop}
   f(\ball{x}{r})\supseteq\ball{f(x)}{\sigma r}\cap
   \ball{f(x_0)}{\zeta},  \quad\forall x\in\ball{x_0}{\delta},
   \ \forall r\in [0,\delta].
\end{equation}
\end{definition}

The role of a surjection property in preserving convexity
of sets should not come as a surprise: the convexity
of the image requires indeed line segments joining
points in the image of a set to belong to the image,
that is a certain openness/covering behaviour of the reference
mapping.

It is well known (see, for instance, \cite{DonRoc14,Ioff16,Mord06})
that the property of openness at a linear rate for a mapping $f$
around $x_0$ can be equivalently reformulated as follows:
there exist positive reals $\delta$ and $\kappa$ such that
\begin{equation}   \label{in:pointmr}
   \dist{x}{f^{-1}(y)}\le\kappa d(y,f(x)),\quad\forall x\in
   \ball{x_0}{\delta},\ \forall y\in\ball{f(x_0)}{\delta}.
\end{equation}
Whenever the inequality $(\ref{in:pointmr})$ holds, $f$ is
said to be {\it metrically regular} around $x_0$. The infimum
over all values $\kappa$ for which there exists $\delta>0$
such that $(\ref{in:pointmr})$ holds true is called
{\it exact regularity bound} of $f$ around $x_0$ and it will
be denoted by $\reg{f}{x_0}$, with the convention that
$\reg{f}{x_0}=+\infty$ means that $f$ fails to be metrically
regular around $x_0$.

\begin{remark} \label{rem:pointlopsimpl}
(i) It is convenient to note that, whenever $f$ is continuous
at $x_0$, the inclusion defining the
openness of $f$ at a linear rate around $x_0$ takes the
simpler form: there exists positive $\delta$ and $\sigma$
such that
\begin{equation}   \label{in:pointlopsimpl}
   f(\ball{x}{r})\supseteq\ball{f(x)}{\sigma r},
   \quad\forall x\in\ball{x_0}{\delta},
   \ \forall r\in [0,\delta].
\end{equation}

(ii) From the inclusion $(\ref{in:pointlopsimpl})$ it is clear
that, whenever a mapping $f$ is open at a linear rate around
$x_0$ and continuous at the same point, it holds
\begin{equation}     \label{in:intelop}
   f(\inte S)\subseteq\inte f(S),
\end{equation}
provided that $S\subseteq\ball{x}{\delta}$, where $\delta$ is as
above. Indeed, if it is $x\in\inte S$, then for some $r\in (0,\delta)$
it must be $\ball{x}{r}\subseteq S$. Therefore, one gets
$$
  \ball{f(x)}{\sigma r}\subseteq f(\ball{x}{r})\subseteq f(S).
$$
In turn, from the inclusion $(\ref{in:intelop})$, one deduces
$$
  f^{-1}(y)\cap S\subseteq\fr S,\quad\forall y\in\fr f(S).
$$
\end{remark}

As the behaviour formalized by openness at a linear rate/metric
regularity plays a crucial role in a variety of topics in
variational analysis, it has been widely investigated in the past
decades and several criteria for detecting the occurrence of it
are now at disposal. In the case of smooth mappings between
Banach spaces, the main criterion for openness at a linear rate/metric
regularity, known under the name of Lyusternik-Graves theorem,
can be stated as follows (see \cite{DonRoc14,Ioff16,Mord06}).

\begin{theorem}[Lyusternik-Graves]
Let $f:\X\longrightarrow\Y$ be a mapping between Banach spaces.
Suppose that $f$ is strictly differentiable at $x_0\in\X$.
Then, $f$ is open at a linear rate around $x_0$ iff $\der{f}{x_0}$
is onto, i.e. $\der{f}{x_0}(\X)=\Y$.
\end{theorem}

The above criterion is usually complemented with the following
(primal and dual) estimates of the exact regularity bound,
which are relevant for the present analysis:
$$
  \reg{f}{x_0}=\sup_{\|y\|\le 1}\inf\{\|x\|:\ x\in
  \der{f}{x_0}^{-1}(y)\}
$$
and
$$
  \reg{f}{x_0}=\left(\inf_{\|y^*\|_*=1}\|\der{f}{x_0}^*y^*\|_*
  \right)^{-1}=\left(\dist{\nullv^*}{\der{f}{x_0}^*(\Sfer^*)}
  \right)^{-1},
$$
where $\Lambda^*\in\mathcal{L}(\Y^*,\X^*)$ denotes the adjoint
operator to $\Lambda\in\mathcal{L}(\X,\Y)$ and the conventions
$$
   \inf\varnothing=+\infty \qquad\hbox{ and }\qquad
   1/0=+\infty
$$
are adopted. Remember that $\Lambda\in\mathcal{L}(\X,\Y)$ is
onto iff $\Lambda^*$ has bounded inverse. It is worth noting
that, when both $\X$ and $\Y$ are finite-dimensional Euclidean
spaces, the condition on $\der{f}{x_0}$ to be onto reduces to the
fact that Jacobian matrix of $f$ at $x_0$ is full-rank.
Furthermore, whenever $\der{f}{x_0}$ happens to be invertible,
one has $\reg{f}{x_0}=\|\der{f}{x_0}^{-1}\|_\mathcal{L}$.


\section{An extension of the Polyak convexity principle}  \label{Sec:3}

Given $c>0$, let us introduce the following subclasses of uniformly
convex subsets of $\X$, with modulus of convexity of power type $2$:
$$
  \Uniconvc(\X)=\{S\subseteq\X:\ \delta_S(\epsilon)\ge c\epsilon^2,
  \ \forall\epsilon\in (0,\diamet S)\}
$$
and
$$
  \Uniconv(\X)=\bigcup_{c>0}\Uniconvc(\X).
$$

\begin{remark}     \label{rem:midpointconv}
In the proof of the next theorem the following fact, which can
be easily proved by an iterative bisection procedure, will be
used:
any closed subset $V$ of a Banach space is convex iff ${y_1+y_2
\over 2}\in V$, whenever $y_1,\, y_2\in V$.
It is easy to see that if $V$ is not closed, this mid-point
property does not imply the convexity of $V$. Consider, for
instance, the set $V$ defined by
$$
  V=\bigcup_{k=0}^\infty\left\{{i\over 2^k}:\ i\in
  \{0,1,2,3,\dots, 2^k\}\right\}\subseteq [0,1].
$$
Since $V$ is countable, as a countable union of finite sets,
it is strictly included in $[0,1]$. Therefore $V$ can not be
convex, because it contains $0$ and $1$, even though it has
the mid-point property, as one checks without difficulty.
\end{remark}

Below, the main result of the paper is established.

\begin{theorem}   \label{thm:extPCP}
Let $f:\Omega\longrightarrow\Y$ be a mapping between Banach
spaces, with $\Omega$ open nonempty subset of $\X$.
Let $x_0\in\Omega$ and $c>0$ such that:

\begin{itemize}

\item[(i)] $f\in\Coneone(\inte\ball{x_0}{r_0})$, for some
$r_0>0$;

\item[(ii)] $\der{f}{x_0}$ is onto;

\item[(iii)] it holds
$$
  {\reg{f}{x_0}\cdot\lip{\dif{f}}{\inte\ball{x_0}{r_0}}\over 8}
  <c.
$$

\end{itemize}
Then, there exists $\rho\in (0,r_0)$ such that, for every $S\in
\Uniconvc(\X)$, with $S\subseteq\ball{x_0}{\rho}$ and $f(S)$
closed, it is $f(S)\in\Uniconv(\Y)$.
\end{theorem}

\begin{proof}
The proof is divided into two parts.

\noindent {\it First part}: Let us show that $f(S)$ is convex.
According to the hypothesis (iii), it is possible to fix positive
reals $\kappa$ and $\ell$ in such a way that $\kappa>\reg{f}{x_0}$,
$\ell>\lip{\dif{f}}{\inte\ball{x_0}{r_0}}$, and the following
inequality is fulfilled
\begin{equation}    \label{in:klc}
      {\kappa\ell\over 8}<c.
\end{equation}
By virtue of hypotheses (i) and (ii), as $f$ is in particular strictly
differentiable at $x_0$, it is possible to invoke the
Lyusternik-Graves theorem, ensuring that $f$ is metrically regular
around $x_0$.
This means that there exist positive reals $\tilde\kappa$ and
$\tilde r$ such that
$$
   \reg{f}{x_0}<\tilde\kappa<\kappa, \qquad\qquad
   \tilde r\in (0,r_0),
$$
and
\begin{equation}    \label{in:mrfx0}
   \dist{x}{f^{-1}(y)}\le\tilde\kappa\|y-f(x)\|,
   \quad\forall x\in\ball{x_0}{\tilde r},\
   \forall y\in\ball{f(x_0)}{\tilde r}.
\end{equation}
Besides, by the continuity of $f$ at $x_0$, corresponding to
$\tilde r$ there exists $r_*\in (0,r_0)$ such that
$$
   f(x)\in\ball{f(x_0)}{\tilde r},\quad\forall x\in
   \ball{x_0}{r_*}.
$$
Then, take $\rho\in (0,\min\{\tilde r,\, r_*\})$.
Notice that, in the light of Remark \ref{rem:pointlopsimpl},
up to a further reduction in the value of $\rho$, one can
assume that for some $\sigma>0$ it holds
\begin{equation}     \label{in:lopx0rho}
  f(\ball{x}{r})\supseteq\ball{f(x)}{\sigma r},\quad\forall
  x\in\ball{x_0}{\rho},\ \forall r\in [0,\rho].
\end{equation}
Now, take an arbitrary element $S\in\Uniconvc(\X)$, with
$S\subseteq\ball{x_0}{\rho}$ and such that $f(S)$ is closed.
According to Remark \ref{rem:midpointconv},
the convexity of $f(S)$ can be proved by showing that
for every $y_1,\, y_2\in f(S)$, with $y_1\ne y_2$, it holds
${y_1+y_2\over 2}\in f(S)$. To this aim, let $x_1,\, x_2
\in S$ be such that $y_1=f(x_1)$ and $y_2=f(x_2)$. For
convenience, set

$$
  \bar x={x_1+x_2\over 2}\qquad\hbox{ and }\qquad
  \bar y={y_1+y_2\over 2}
$$
Notice that, as $y_1\ne y_2$, it must be also $x_1\ne x_2$.
Moreover, as $S\subseteq\ball{x_0}{\rho}\subseteq
\ball{x_0}{r_*}$, one has $y_1,\, y_2\in\ball{f(x_0)}{\tilde r}$
and therefore, by the convexity of a ball, one has also
$\bar y\in\ball{f(x_0)}{\tilde r}$.
Thus, since $\bar x\in\ball{x_0}{\tilde r}$ and $y\in
\ball{f(x_0)}{\tilde r}$, then inequality $(\ref{in:mrfx0})$
implies
\begin{equation}   \label{in:mrfbarx}
  \dist{\bar x}{f^{-1}(\bar y)}\le\tilde\kappa
  \|\bar y-f(\bar x)\|.
\end{equation}
If $\bar y=f(\bar x)$ the proof of the convexity of $f(S)$
is complete, because $\bar x\in S$.
Otherwise, it happens that $\|\bar y-f(\bar x)\|>0$, so the
inequality $(\ref{in:mrfbarx})$ entails the existence of $\hat x
\in f^{-1}(\bar y)$ such that
$$
   \|\hat x-\bar x\|<\kappa\|\bar y-f(\bar x)\|.
$$
By taking account of the estimate $(\ref{in:2ndordest})$
in Remark \ref{rem:smoothfacts} (i), as
it is $[x_1,x_2]\in\ball{x_0}{\rho}\subseteq\inte\ball{x_0}{r_0}$,
one consequently obtains
$$
    \|\hat x-\bar x\|<\kappa{\ell\over 8}\|x_1-x_2\|^2,
$$
that is $\hat x\in\ball{\bar x}{{\kappa\ell\over 8}\|x_1-x_2\|^2}$.
Since $S\in\Uniconvc(\X)$ and the inequality $(\ref{in:klc})$ is in force,
in the light of what observed in Remark \ref{rem:strconvfacts} (iii)
it follows
$$
  \ball{\bar x}{{\kappa\ell\over 8}\|x_1-x_2\|^2}\subseteq S,
$$
with the consequence that $\hat x\in S$ and hence $\bar y=f(\hat x)$
turns out to belong to $f(S)$.

\noindent {\it Second part}: Let us prove now the assertion in the thesis.
According to what noted in Remark \ref{rem:smoothfacts} (ii), under
the above hypotheses $f(S)$ is bounded. Fix $\epsilon\in (0,\diamet f(S))$
and take arbitrary $y_1,\, y_2\in f(S)$, with $\|y_1-y_2\|=\epsilon$.
Let $\bar y,\, x_1,\, x_2,\, \bar x$ and
$\hat x$ be as in the first part of the proof (it may happen that
$\hat x=\bar x$). In order to prove that $f(S)\in\Uniconv(\Y)$,
it is to be shown that, independently of $y_1,\, y_2\in f(S)$ and
$\epsilon$,
there exists $\gamma>0$ such that $\ball{\bar y}{\gamma\epsilon^2}
\subseteq f(S)$.
Again recalling Remark \ref{rem:smoothfacts} (ii), it is possible to
define the positive real value
$$
  \beta=\sup_{x\in S}\|\der{f}{x}\|_\mathcal{L}+1<+\infty.
$$
By virtue of inequality $(\ref{in:klc})$, it is possible to pick
$\eta\in (0,c-{\kappa\ell\over 8})$ in such a way that
$$
  \hat x\in\ball{\bar x}{{\kappa\ell\over 8}\|x_1-x_2\|^2}
  \subseteq \ball{\bar x}{\left({\kappa\ell\over 8}+\eta\right)
  \|x_1-x_2\|^2}\subseteq S.
$$
From the last chain of inclusions, it readily follows that
$$
  \ball{\hat x}{\eta\|x_1-x_2\|^2}\subseteq S.
$$
Since, by the mean-value theorem, it is
$$
  \|y_1-y_2\|\le\beta\|x_1-x_2\|,
$$
one obtains
$$
  \epsilon^2=\|y_1-y_2\|^2\le\beta^2\|x_1-x_2\|^2,
$$
and hence $\ball{\hat x}{\eta\epsilon^2/\beta^2}
\subseteq S$. Now, recall that $f$ is open at a linear rate around
$x_0$. Accordingly, as $S\subseteq\ball{x_0}{\rho}$, up to a further
reduction in the value of $\eta>0$ in such a way that
$\eta\diamet^2 f(S)/\beta^2<\rho$, one finds
$$
  \ball{\bar y}{\sigma\eta{\epsilon^2\over\beta^2}}\subseteq
  f\left(\ball{\hat x}{\eta{\epsilon^2\over\beta^2}}\right)
  \subseteq f(S)
$$
(remember the inclusion $(\ref{in:lopx0rho})$).
Thus, since by construction $\sigma$, $\eta$ and $\beta$ are
independent of $y_1,\, y_2$ and $\epsilon$, one can conclude
that
$$
  \delta_{f(S)}(\epsilon)\ge{\sigma\eta\over\beta^2}\epsilon^2.
$$
By arbitrariness of $\epsilon\in (0,\diamet f(S))$, this
completes the proof.
\end{proof}

A first comment to Theorem \ref{thm:extPCP} concerns its hypothesis
(iii), which seems to find no counterpart in the convexity principle
due to B.T. Polyak (see \cite[Theorem 2.1]{Poly01}). Such hypothesis
postulates a uniform convexity property of $S$, which must be quantitatively
adequate to the metric regularity of $f$ and to the Lipschitz
continuity of $\dif{f}$ around $x_0$.
Matching this condition is guaranteed for strongly convex sets
(in particular, for balls) with a sufficiently small radius,
provided that the underlying Banach space fulfils a certain
uniform convexity assumption.
This fact is clarified by the following

\begin{corollary}      \label{cor:strconvpcp}
Let $f:\Omega\longrightarrow\Y$ be a mapping between Banach
spaces, with $\Omega$ open nonempty subset of $\X$.
Let $x_0\in\Omega$ be such that:

\begin{itemize}

\item[(i)] $(\X,\|\cdot\|)$ admits a modulus of convexity
    of power type 2;

\item[(ii)] $f\in\Coneone(\inte\ball{x_0}{r_0})$, for some
$r_0>0$;

\item[(iii)] $\der{f}{x_0}$ is onto.

\end{itemize}

\noindent Then, there exists $\rho\in (0,r_0)$ such that, for every
$r$-convex set $S$, with $r\in [0,\rho)$ and $f(S)$ closed, it holds
$f(S)\in\Uniconv(\Y)$.
\end{corollary}

\begin{proof}
By virtue of the hypothesis (i), according to Example \ref{ex:ucset}
(ii), any $r$-convex set $S$ belongs to $\Uniconv(\X)$, for every
$r>0$. More precisely, on account of the inequality $(\ref{in:ucmscs})$,
one has
$$
   \delta_{S}(\epsilon)\ge r\delta_{\X}
   \left({\epsilon\over r}\right)\ge{\gamma\over r}\epsilon^2,
   \quad\forall\epsilon\in (0,2r],
$$
for some $\gamma>0$. Therefore, in order for the hypothesis (iii) of
Theorem \ref{thm:extPCP} to be satisfied, it suffices to take
$$
   r<{8\gamma\over \reg{f}{x_0}\cdot\lip{\dif{f}}{\inte\ball{x_0}{r_0}}+1}.
$$
Then, the thesis follows from Theorem \ref{thm:extPCP}.
\end{proof}

On the other hand, notice that Theorem \ref{thm:extPCP} does not
make any direct assumption on the Banach space $(\X,\|\cdot\|)$
(nonetheless, take into account what remarked at the end of
Example \ref{ex:ucset} (i)).
Furthermore, since any ball $\ball{x_0}{r}$ is a $r$-convex sets,
it should be clear that Corollary \ref{cor:strconvpcp} allows one
to embed in the current theory the Polyak convexity principle and
its refinement \cite[Theorem 3.2]{Uder13}.

Another comment to Theorem \ref{thm:extPCP} deals with the topological
assumption on the image $f(S)$. Of course, whenever $\X$ is a
finite-dimensional Euclidean space, $f(S)$ is automatically closed,
because $S$ is compact and $f$ is continuous on $S$.
In an infinite-dimensional setting, the same issue becomes subtler.
The closedness assumption thus appears also in the formulation of
other results for the convexity of images of mappings between
infinite-dimensional spaces (see \cite[Theorem 2.2]{BacSam09}).
It is clear that, whenever $\der{f}{x_0}$ not only is onto but,
in particular, is invertible, $f$ turns out to be a diffeomorphism
around $x_0$. As a consequence, for a proper $r_0>0$, any closed
set $S\subseteq\ball{x_0}{r_0}$ has a closed image. Nevertheless,
in the general setting of Theorem \ref{thm:extPCP}, to the best of the
author's knowledge, the question of formulating sufficient conditions
on $f$ in order for $f(S)$ to be closed is still open. The next
proposition, which is far removed from providing a solution to such
a question, translates the topological assumption on the image $f(S)$
into variational terms.

\begin{proposition}     \label{pro:closim}
Let $f:\Omega\longrightarrow\Y$ be a mapping between Banach
spaces, with $\Omega$ open nonempty subset of $\X$, and let
$x_0\in\Omega$. Suppose that:

\begin{itemize}

\item[(i)] $f$ is continuous in $\ball{x_0}{r_0}$,
for some $r_0>0$;

\item[(ii)] the function $x\mapsto \dist{x}{f^{-1}(y)}$ is
weakly lower semicontinuous, for every
$y\in\ball{f(x_0)}{r_0}$;

\item[(iii)] $(\X,\|\cdot\|)$ is reflexive;

\item[(iv)] $f$ is metrically regular around $x_0$.

\end{itemize}
Then, there exists $\rho\in (0,r_0)$ such that, for every
closed convex set $S\subseteq\ball{x_0}{\rho}$, $f(S)$
is closed.
\end{proposition}

\begin{proof}
Since by the hypothesis (iv) $f$ is metrically regular
around $x_0$, there exist positive real $r\in (0,r_0)$ and
$\kappa$ such that
\begin{equation}     \label{in:mrfr}
   \dist{x}{f^{-1}(y)}\le\kappa\|f(x)-y\|,\quad\forall
   x\in\ball{x_0}{r},\ \forall y\in \ball{f(x_0)}{r}.
\end{equation}
By the continuity of $f$ at $x_0$, there exists
$\rho\in (0,r)$ such that
$$
   f(x)\in\ball{f(x_0)}{r},\quad\forall
    x\in\ball{x_0}{\rho}.
$$
Thus, whenever $S\subseteq\ball{x_0}{\rho}$, one has
$f(S)\subseteq\ball{f(x_0)}{r}$.

Now, suppose that $S\subseteq\ball{x_0}{\rho}$ is a closed
convex set and take an arbitrary $y\in\cl f(S)\subseteq
\ball{f(x_0)}{r}$. Let $(y_n)_n$ be a sequence in $f(S)$, such that
$y_n\longrightarrow y$ as $n\to\infty$. As $y_n\in f(S)$, there exists a
sequence $(x_n)_{n}$ in $S$ such that $y_n=f(x_n)$, for
each $n\in\N$. Notice that, since $x_n\in S\subseteq\ball{x_0}{\rho}
\subseteq\ball{x_0}{r}$ and $y\in\cl f(S)\subseteq
\ball{f(x_0)}{r}$, the inequality $(\ref{in:mrfr})$ applies,
namely
\begin{equation}     \label{in:mrfrxnyn}
   \dist{x_n}{f^{-1}(y)}\le\kappa\|f(x_n)-y\|=
   \kappa\|y_n-y\|,\quad\forall n\in\N.
\end{equation}
This shows that $\dist{x_n}{f^{-1}(y)}\longrightarrow 0$ as
$n\to\infty$ and therefore
$$
  \inf_{x\in S}\dist{x}{f^{-1}(y)}=0.
$$
As a closed convex set, $S$ is also weakly closed. Moreover,
as a bounded subset of a reflexive Banach space, $S$ is
weakly compact.
Thus, since $y\in\ball{f(x_0)}{r_0}$, by virtue of the hypothesis
(ii), there must exist $\tilde x\in S$ such that
$$
  \dist{\tilde x}{f^{-1}(y)}=0.
$$
Since $f$ is continuous, the last inequality entails that
$\tilde x\in f^{-1}(y)$. This leads to conclude that $y\in f(S)$, thereby
completing the proof.
\end{proof}

The hypothesis (ii) in Proposition \ref{pro:closim} happens to be
always satisfied if $f$ is a linear mapping. In the nonlinear case,
the situation is expected to be much more complicate.

Let $C\subseteq\Y$ be a closed convex cone with apex at $\nullv$
and let $S\subseteq\X$ be nonempty and convex. Recall that a mapping
$f:S\longrightarrow\Y$ is said to be {\it convex-like} on $S$ with respect
to $C$ if for every $x_1,\, x_2\in S$ and $t\in [0,1]$, there
exists $x_t\in S$ such that
$$
  (1-t)f(x_1)+tf(x_2)\in f(x_t)+C.
$$
Convex-likeness is a generalization of the notion of $C$-convexity
of mappings taking values in partially ordered vector spaces.
It should be evident that, when $\Y=\R$, $C=[0,+\infty)$ and $x_t=(1-t)
x_1+tx_2$, the above inclusion reduces to the well-known inequality
defining the convexity of a functional. The class of convex-like
mappings has found a large employment in optimization and related
topics. For instance, if $\R^m$ and $C=\R^m_+$ it is readily seen that
this class includes all mappings $f=(f_1,\dots,f_m)$,
having each component $f_i:S\longrightarrow\R$, $i=1,\dots,m$ convex on a
convex set. For a detailed discussion about the notion of
convex-likeness of mappings, its variants and their impact
on the study of variational problems, the reader can refer to
\cite{MasRap00}.
The next corollary, which can be achieved as a direct consequence
of Theorem \ref{thm:extPCP}, reveals that any $\Coneone$ smooth
mapping behave as a convex-like mapping on uniformly convex
sets of class $\Uniconvc(\X)$ near a regular point.

\begin{corollary}
Let $f:\Omega\longrightarrow\Y$ be a mapping between Banach
spaces, $x_0\in\Omega$ and $c>0$. If $f$, $x_0$ and $c$
satisfy all hypotheses of Theorem \ref{thm:extPCP}, then
there exists $\rho>0$ such that, for every $S\in\Uniconvc(\X)$,
with $S\subseteq\ball{x_0}{\rho}$ and $f(S)$ closed, and every
cone $C\subseteq\Y$, the mapping $f:S\longrightarrow\Y$ is
convex-like on $S$ with respect to $C$.
\end{corollary}

\begin{proof}
The thesis follows at once by Theorem \ref{thm:extPCP}, from being
$$
   (1-t)f(x_1)+tf(x_2)\in f(S)\subseteq f(S)+C, \quad
   \forall x_1,\, x_2\in S,\ \forall t\in [0,1].
$$
\end{proof}


\section{Applications to optimization}   \label{Sec:4}

Throughout this section, applications of Theorem \ref{thm:extPCP}
will be considered to the study of constrained optimization problems,
having the following format
$$
   \min_{x\in S}\varphi(x)\quad\hbox{ subject to }\quad g(x)\in C,
   \leqno ({\mathcal P})
$$
where $\varphi:\X\longrightarrow\R$ and $g:\X\longrightarrow\Y$ are
given functions between Banach spaces, $S\subseteq\X$ and $C\subseteq\Y$
are given (nonempty) closed and convex sets.
Such a format is frequently employed in the literature for subsuming
under a general treatment a broad spectrum of finite and infinite-dimensional
extremum problems, with various kinds of constraints. The feasible region of
problem $({\mathcal P})$ will be henceforth denoted by $R$, i.e.
$R=S\cap g^{-1}(C)$.

According to a long-standing approach in optimization, now recognized as ISA (acronym
standing for Image Space Analysis), the analysis of several
issues related to problem $({\mathcal P})$ can be performed by associating
with $({\mathcal P})$ and with an element $x_0\in R$ the mapping
$\immap{x_0}:\X\longrightarrow\R\times\Y$, which is defined by
$$
  \immap{x_0}(x)=(\varphi(x)-\varphi(x_0),g(x))
$$
(see, for instance, \cite{Gian05} and references therein).
It is natural to believe that the mapping $\immap{x_0}$ inherits
certain structural features of the given problem.
Such issues as the solution existence, optimality conditions, duality,
and so on, can be investigated by studying relationships between
the two subsets of the space $\R\times\Y$, namely $\immap{x_0}(S)$ and
$Q=(-\infty,0)\times C$, associated with $({\mathcal P})$.

\begin{remark}      \label{rem:ISAlop}
Directly from the above constructions, it is possible to prove the
following well-known facts:

(i) $x_0\in R$ is a global solution to $({\mathcal P})$ iff $\immap{x_0}(S)
\cap Q=\varnothing$;

(ii) $x_0\in R$ is a local solution to $({\mathcal P})$ iff there exists
$r>0$ such that $\immap{x_0}(S\cap\ball{x_0}{r})\cap Q=\varnothing$.
\end{remark}

The above facts have been largely employed as a starting point for
formulating optimality conditions within ISA. Another relevant
property connected with optimality is openness at a linear rate.
Its presence, indeed, has been observed to be in contrast with
optimality (see, for instance, the so-called noncovering principle
in \cite{Ioff16}). Below, a lemma related to this phenomenon, which
will be exploited in the proof of the next result, is presented
in full detail.

\begin{lemma}    \label{lem:lopnotsol}
With reference to a problem $({\mathcal P})$, suppose that the mapping
$\immap{x_0}$ is open at a linear rate around $x_0\in R$ and $x_0\in
\inte S$. Then, $x_0$ is not a local solution to $({\mathcal P})$.
\end{lemma}

\begin{proof}
By the hypothesis, according to Definition \ref{def:lopmap} there exist
positive constants $\delta$, $\zeta$, and $\sigma$ such that, if taking
in particular $x=x_0$ in inclusion $(\ref{in:pointlop})$, it holds
$$
  \immap{x_0}(\ball{x_0}{r})\supseteq\ball{\immap{x_0}(x_0)}{\sigma r}
  \cap\ball{\immap{x_0}(x_0)}{\zeta},\quad\forall r\in [0,\delta].
$$
Notice that, if $r<\zeta/\sigma$, then the above inclusion reduces to
\begin{equation}    \label{in:prolopx0}
   \immap{x_0}(\ball{x_0}{r})\supseteq\ball{\immap{x_0}(x_0)}{\sigma r}
   =\ball{(0,g(x_0))}{\sigma r}.
\end{equation}
Since $x_0\in\inte S$, there exists $r_0>0$ such that $\ball{x_0}{r_0}
\subseteq S$.
Now, fix an arbitrary $r\in (0,\, \min\{r_0,\, \zeta/\sigma\})$ and
pick $t\in (0,\sigma r)$. Then, on the account of inclusion $(\ref{in:prolopx0})$,
there exists $x_r\in\ball{x_0}{r}$ such that
$$
  \immap{x_0}(x_r)=(-t,g(x_0))\in\ball{(0,g(x_0))}{\sigma r},
$$
that is
$$
  \varphi(x_r)-\varphi(x_0)=-t<0 \qquad\hbox{ and }\qquad
  g(x_r)=g(x_0)\in C.
$$
This means that $x_r\in S\cap g^{-1}(C)$ and $\varphi(x_r)<\varphi(x_0)$,
what contradicts the local optimality of $x_0$ for  $({\mathcal P})$,
by arbitrariness of $r$.
\end{proof}

The next theorem, which extends a similar result established in
\cite[Theorem 3.2]{Uder13}, provides an answer to the question of
solution existence for problem $({\mathcal P})$ and, at the same
time, furnishes an optimality condition for detecting a solution.
In order to formulate such a theorem, let us denote by $\ncone{C}{\bar y}=
\{y^*\in\Y^*:\ \langle y^*,y-\bar y\rangle\le 0,\quad\forall y\in C\}$
the normal cone to $C$ at $\bar y$ in the sense of convex analysis.
Besides, let us denote by $\Lagr:\Y^*\times\X\longrightarrow\R$ the
Lagrangian function associated with problem $({\mathcal P})$, i.e.
$$
   \Lagr (y^*,x)=\varphi(x)+\langle y^*,g(x)\rangle.
$$
The proof, whose main part is given for the sake of completeness,
adapts an argument already exploited in \cite{Uder13}. It derives
solution existence from the weak compactness of the problem
image and the optimality condition by a linear separation technique.
In both the cases, convexity is the geometrical property that
makes this possible.

\begin{theorem}     \label{thm:constoptp}
Given a problem $({\mathcal P})$, let $x_0\in g^{-1}(C)$
and let $c$ be a positive real. Suppose that:

\begin{itemize}

\item[(i)] $(\Y,\|\cdot\|)$ is a reflexive Banach space;

\item[(ii)] $\varphi,\, g\in\Coneone(\inte\ball{x_0}{r_0})$,
for some $r_0>0$ and $\der{\immap{x_0}}{x_0}$ is onto;

\item[(iii)] it holds
\begin{equation}     \label{in:relipuconvcond}
  {\reg{\immap{x_0}}{x_0}\cdot\lip{\dif{\immap{x_0}}}{\inte\ball{x_0}{r_0}}\over 8}
  <c.
\end{equation}
\end{itemize}
Then, there exists $\rho\in (0,r_0)$ such that, for every $S\in
\Uniconvc(\X)$, with $x_0\in\inte S\subseteq\ball{x_0}{\rho}$ and $\immap{x_0}(S)$
closed, one has
\begin{itemize}

\item[(t)] there exists a global solution $\bar x_S\in R$ to $({\mathcal P})$;

\item[(tt)] $\bar x_S\in\fr S$ and hence $\bar x_S\in\fr R$;

\item[(ttt)] there exists $y^*_S\in\ncone{C}{g(\bar x_S)}$ such that
$$
   \Lagr(y^*_S,\bar x_S)=\min_{x\in S}\Lagr(y^*_S,x).
$$
\end{itemize}
\end{theorem}

\begin{proof}
(t) Under the hypotheses (ii) and (iii), one can apply Theorem \ref{thm:extPCP}.
If $\rho>0$ is as in the thesis Theorem \ref{thm:extPCP}, fix a set $S\in
\Uniconvc(\X)$ satisfying all requirements in the above statement. Then
its image $\immap{x_0}(S)$ turns out to be a convex, closed and bounded subset
of $\R\times\Y$, with nonempty interior. The existence of a global solution
to $({\mathcal P})$ will be achieved by proving that an associated minimization
problem in the space $\R\times\Y$ does admit a global solution.
To do so, define
$$
   \tau=\inf\{t:\ (t,y)\in \immap{x_0}(S)\cap Q\}.
$$
Notice that $x_0\in R$. Since $\der{\immap{x_0}}{x_0}$ is onto, by the
Lyusternik-Graves theorem the mapping $\immap{x_0}$ too is open at a
linear rate around $x_0$. Thus, since $x_0\in\inte S$, in the light
of Lemma \ref{lem:lopnotsol} $x_0$ must fail to be a local (and hence,
a fortiori, global) solution to $({\mathcal P})$. Consequently,
according to what observed in Remark \ref{rem:ISAlop} (i), it must be
$$
  \immap{x_0}(S)\cap Q\ne\varnothing.
$$
This implies that $\tau<+\infty$. Furthermore, if setting
\begin{equation}    \label{eq:deftauisa}
   \bar\tau=\inf\{t:\ (t,y)\in \immap{x_0}(S)\cap \cl Q\},
\end{equation}
it is possible to see that actually it is $\bar\tau=\tau$. Indeed,
since $x_0$ is not a solution to $({\mathcal P})$, there exists
$\hat x\in R$ such that $\varphi(\hat x)-\varphi(x_0)<0$, and so
$\immap{x_0}(\hat x)=(\varphi(\hat x)-\varphi(x_0),g(\hat x))
\in\immap{x_0}(S)\cap Q$. As $\immap{x_0}(S)\cap Q\subseteq
\immap{x_0}(S)\cap \cl Q$, it follows that $\bar\tau\le\tau\le
\varphi(\hat x)-\varphi(x_0)<0$. Hence, for any $\epsilon\in
(0,-\bar\tau)$ there exists $(t_\epsilon,y_\epsilon)\in\immap{x_0}(S)
\cap \cl Q$ such that $t_\epsilon<\bar\tau+\epsilon<0$. Noting
that $\cl Q=(-\infty,0]\times C$, this implies that
$(t_\epsilon,y_\epsilon)\in\immap{x_0}(S)\cap Q$ and
consequently that $\bar\tau\le\tau\le t_\epsilon<\bar\tau+\epsilon<0$.
Letting $\epsilon\to 0^+$, one obtains $\bar\tau=\tau$.

Now, as the set $\immap{x_0}(S)$ is closed, convex and bounded, so
is its subset $\immap{x_0}(S)\cap \cl Q$.  The boundedness of the
latter implies that $\bar\tau>-\infty$. Moreover, by virtue of the
hypothesis (i), $\immap{x_0}(S)\cap \cl Q$ turns out to be weakly
compact. Since the projection mapping $\Pi_\R:\R\times\Y\longrightarrow\R$,
given by $\Pi_\R(t,y)=t$ is continuous and convex, it is also weakly
l.s.c., with the consequence that the infimum defined in $(\ref{eq:deftauisa})$
is actually attained at some $(\bar t,\bar y)\in\immap{x_0}(S)\cap \cl Q$.
This means that there exists $\bar x_S\in S$ such that
$$
   \tau=\bar\tau=\bar t=\varphi(\bar x_S)-\varphi(x_0)\quad
   \hbox{ and }\bar y=g(\bar x_S) \in C.
$$
Let us show that $\bar x_S$ is a global solution to $({\mathcal P})$.
Assume to the contrary that there is $\hat x\in R$ such that $\varphi
(\hat x)<\varphi(\bar x_S)$. Then, one finds
\begin{eqnarray*}
  \hat t &=& \varphi(\hat x)-\varphi(x_0)=\varphi(\hat x)-\varphi(\bar x_S)
  +\varphi(\bar x_S)-\varphi(x_0) \\
  &<& \varphi(\bar x_S)-\varphi(x_0)=\bar t=\bar\tau=\tau.
\end{eqnarray*}
Since it is $\hat x\in R$, then $\hat x\in S$ and $\hat y=g(\hat x)\in C$,
wherefrom one has $(\hat t,\hat y)\in\immap{x_0}(S)\cap Q$, which
contradicts the definition of $\tau$.

(tt) To prove that $\bar x_S$ belongs to $\fr S$, notice that
$(\bar t,\bar y)=\immap{x_0}(\bar x_S)\in\fr\immap{x_0}(S)$.
Then, by recalling what mentioned in Remark \ref{rem:pointlopsimpl} (ii),
this assertion follows from the openness at a linear rate of
$\immap{x_0}$ around $x_0$.

(ttt) Again remembering Remark \ref{rem:ISAlop} (i), by the global
optimality of $\bar x_S$, it results in
\begin{equation}    \label{eq:barxSopt}
  \immap{\bar x_S}(S)\cap Q=\varnothing.
\end{equation}
As one readily checks, it holds
$$
  \immap{\bar x_S}(S)=\immap{x_0}(S)+(\varphi(x_0)-
  \varphi(\bar x_S),\nullv),
$$
that is to say $\immap{\bar x_S}(S)$ is a translation of
$\immap{x_0}(S)$. Therefore, $\immap{\bar x_S}(S)$ too is
a closed, bounded, convex subset of $\R\times\Y$, with nonempty
interior. Since $(\ref{eq:barxSopt})$ is true, the Eidelheit theorem
makes it possible to linearly separate $\immap{\bar x_S}(S)$ and
$\cl Q$. In other terms, this means the existence of a pair
$(\gamma,y^*)\in(\R\times\Y)\backslash \{(0,\nullv^*)\}$ and
$\alpha\in\R$ such that
\begin{equation*}
     \gamma(\varphi(x)-\varphi(\bar x_S))+\langle y^*,g(x)
     \rangle\ge\alpha,\quad\forall x\in S,
\end{equation*}
and
\begin{equation*}
  \gamma t+\langle y^*,y\rangle\le\alpha,\quad\forall
  (t,y)\in\cl Q=(-\infty,0]\times C.
\end{equation*}
The rest of the proof relies on a standard usage of the last
inequalities and does not need to devise any specific adaptation.
\end{proof}

Theorem \ref{thm:constoptp} describes the local behaviour of
a nonlinear optimization problem $({\mathcal P})$ near
a point $x_0\in (\inte S)\cap g^{-1}(C)$, around which the condition
$(\ref{in:relipuconvcond})$ linking the modulus of convexity
of $S$,  the regularity behaviour of $\immap{x_0}$ and the
Lipschitz continuity of its derivative happens to be satisfied:
$({\mathcal P})$ admits a global solution, which lies at the
boundary of the feasible region and can be detected by minimizing
the Lagrangian function. The reader should notice that globality
of a solution and its characterization as a minimizer of a the
Lagrangian function are phenomena typically occurring in convex
optimization. Instead, they generally fail to occur in nonlinear
optimization, where optimality conditions are usually only necessary
or sufficient, and frequently expressed in terms of Lagrangian stationary
by means of first-order derivative.

Another typical phenomenon arising in convex optimization is
the vanishing of the duality gap, i.e. the vanishing of the
value
$$
   \gap({\mathcal P})=\inf_{x\in S}\sup_{y^*\in C^{{}^\ominus}}
   \Lagr(y^*,x)-\sup_{y^*\in C^{{}^\ominus}}\inf_{x\in S}\Lagr(y^*,x),
$$
where $C^{{}^\ominus}=\{y^*\in\Y^*:\ \langle y^*,y\rangle\le 0\}$
is the dual cone to $C$. Such a circumstance, which can be proved to
take place in convex programming under proper qualification conditions,
is known as strong (Lagrangian) duality. In the current setting,
it can be readily achieved as a consequence of Theorem \ref{thm:constoptp},
without the need of extra assumptions, apart from the cone structure now
imposed on the set $C$.

\begin{corollary}
Given a problem $({\mathcal P})$, suppose that $C$ is a closed convex
cone. Under the hypothesis of Theorem \ref{thm:constoptp}, it holds
$$
  \gap({\mathcal P})=0
$$
and there exists a pair $(y^*_S,\bar x_S)\in C^{{}^\ominus}\times R$,
which is a saddle point of $\Lagr$, i.e.
$$
   \Lagr(y^*,\bar x_S)\le\Lagr(y^*_S,\bar x_S)\le
   \Lagr(y^*_S,x),\quad\forall (y^*,x)\in C^{{}^\ominus}\times S.
$$
\end{corollary}

\begin{proof}
Let $\bar x_S$ and $y^*_S$ be as in the thesis of Theorem \ref{thm:constoptp}.
Since $C$ is a closed convex cone, $2g(\bar x_S)$ and $\nullv$ belong to $C$.
By recalling that $y^*_S\in\ncone{C}{g(\bar x_S)}$, one has
$$
   \langle y^*_S,y-g(\bar x_S)\rangle\le 0,\quad\forall y\in C.
$$
By replacing $y$ with $2g(\bar x_S)$ and $\nullv$ in last inequality, one
easily shows that $\langle y^*_S,g(\bar x_S)\rangle=0$ and hence
$y^*_S\in C^{{}^\ominus}$. The rest of the thesis then follows at once.
\end{proof}

The above applications of Theorem \ref{thm:extPCP} demonstrate that,
even in the absence of convexity assumptions on the functional data
of problem $({\mathcal P})$, some good phenomena connected with convexity
may still appear.

\begin{example}
With reference to the problem format $({\mathcal P})$, let $\X=\R^2$,
$\Y=\R$, $C=\{0\}$, and let $\varphi:\R^2\longrightarrow\R$ and $g:\R^2
\longrightarrow\R$ be defined respectively by
$$
   \varphi(x)=x_1^2-x_2^2,\qquad g(x)=x_1^2+x_2^2-1.
$$
Take $x_0=(1/\sqrt{2},1/\sqrt{2})\in g^{-1}(0)=\Sfer$ and $S=\ball{x_0}{r}$.
With the above choice of data, the problem falls out of the realm of
convex optimization: the objective function $\varphi$ is evidently not convex
as well as the feasible region $R=S\cap\Sfer$, for every $r>0$.
Throughout the present example, $\R^2$ is supposed to be equipped with
its Euclidean space structure, so that
$$
  \delta_{\R^2}(\epsilon)\ge {\epsilon^2\over 8},\quad\forall
  \epsilon\in (0,2].
$$
Therefore, $S=\ball{x_0}{r}\in\Uniconv(\R^2)$ and, according to
the estimate in $(\ref{in:ucmscs})$, one finds
$$
  \delta_{\ball{x_0}{r}}(\epsilon)\ge r\delta_{\R^2}
  \left({\epsilon\over r}\right)={\epsilon^2\over 8r},
$$
that is $\ball{x_0}{r}\in\Uniconv_{1/8r}(\R^2)$, for every $r>0$.
Clearly, the function $\immap{x_0}:\R^2\longrightarrow\R^2$,
which is given in this case by
$$
   \immap{x_0}(x)=\left(\begin{array}{c}
                     x_1^2-x_2^2 \\
                     x_1^2+x_2^2-1
             \end{array}\right),
$$
satisfies the smoothness hypothesis of Theorem
\ref{thm:constoptp}. In particular, since it is
$$
   \der{\immap{x_0}}{x}=\left(\begin{array}{cc}
                     2x_1 & -2x_2 \\
                     2x_1 & x_2
             \end{array}\right),
$$
it results in
$$
   \reg{\immap{x_0}}{x_0}=\|\der{\immap{x_0}}{x_0}^{-1}\|_\mathcal{L}=
   \left\|{1\over 2\sqrt{2}}\left(\begin{array}{rr}
                     1 & 1 \\
                     -1 & 1
                     \end{array}\right)\right\|_\mathcal{L}=
                     {1\over 2}.
$$
On the other hand, since the mapping $\dif{\immap{x_0}}:\R^2\longrightarrow
\mathcal{L}(\R^2,\R^2)$ is linear in this case, one finds
\begin{eqnarray*}
   \lip{\dif{\immap{x_0}}}{\R^2} &=& \|\dif{\immap{x_0}}\|_\mathcal{L}
   =\max_{u\in\Sfer}\|\der{\immap{x_0}}{u}\|_\mathcal{L}  \\
   &=& \max_{u\in\Sfer}\max_{v\in\Sfer}\|\der{\immap{x_0}}{u}v\|
      =2\sqrt{2}.
\end{eqnarray*}
Consequently, the condition $(\ref{in:relipuconvcond})$ becomes
$$
  {{1\over 2}\cdot 2\sqrt{2}\over 8}<{1\over 8r}.
$$
Thus, for every $r<1/\sqrt{2}$, by virtue of Theorem \ref{thm:constoptp}
assertions ${\rm (t)-(ttt)}$ hold. In particular, it is not difficult to check
(for instance, by means of a level set inspection)
that for every $S=\ball{x_0}{r}$, with
$r<1/\sqrt{2}$, the unique (global) solution $\bar x_S$ of the
related problem lies in $\fr S$. Notice that this fails to be true if
$r>\sqrt{2-\sqrt{2}}=\|(0,1)-x_0\|>1/\sqrt{2}$, in which case the
solution $\bar x_S=(0,1)$ belongs to $\inte\ball{x_0}{r}=\inte S$.
\end{example}

\vskip1cm

\begin{center}
{\sc Acknowledgements}
\end{center}

The author thanks an anonymous referee for valuable remarks, which
help him to considerably improve the quality of his paper; in particular,
for pointing out a crucial gap in the original proof of Theorem
\ref{thm:constoptp} as well as several inaccuracies.

\vskip1cm


\end{document}